\documentclass[a4paper,reqno,11pt]{amsart}
\usepackage{amsmath,amsthm}
\usepackage{amssymb,latexsym}
\usepackage{enumerate}

\numberwithin{equation}{section}

\newtheorem{theorem}{Theorem}[section]

\newtheorem{proposition}[theorem]{Proposition}
\newtheorem{lemma}[theorem]{Lemma}
\newtheorem{definition}[theorem]{Definition}

\newtheorem{The main theorem}[theorem]{The main theorem}

\theoremstyle{definition}

\allowdisplaybreaks
\begin{document}

\title[Convergence in capacity]{Convergence in capacity of plurisubharmonic functions with given boundary values}

\author{Nguyen Xuan Hong, Nguyen Van Trao and Tran Van Thuy}
\address{Department of Mathematics\\ Hanoi National University of Education\\ 136 Xuan Thuy Street, Caugiay District, Hanoi, Vietnam}
\email{xuanhongdhsp@yahoo.com, ngvtrao@yahoo.com and  thuyhum@gmail.com}
\date{}

\maketitle


\renewcommand{\thefootnote}{}



\renewcommand{\thefootnote}{\arabic{footnote}}
\setcounter{footnote}{0}

\date{}

\begin{abstract}
In this paper, we study the convergence in the capacity of sequence of plurisubharmonic functions.
As an application, we prove  stability results for solutions of the complex Monge-Amp\`ere equations.
\end{abstract}

\maketitle

\section{Introduction}
It is well-known that convergence in the sense of distributions of plurisubharmonic functions does not in general imply convergence of their Monge-Amp\`ere measures.  Therefore, it is important to find conditions on sequences of plurisubharmonic functions such that the corresponding Monge-Amp\`ere measures are convergent in the weak* topology. 

Bedford and  Taylor \cite{BT82} introduced and studied in 1982 the $C_n$-capacity of   Borel sets. 
Xing \cite{Xi96} proved in 1996 that the complex Monge-Amp\`{e}re operator is continuous under convergence of bounded plurisubharmonic functions in $C_n$-capacity. 
He gave a sufficient condition for the weak convergence of complex Monge-Amp\`{e}re mass  of bounded plurisubharmonic functions.
Later,  Xing \cite{Xing} studied in 2008 the convergence in the $C_n$-capacity of a sequence of plurisubharmonic functions in the class $\mathcal F^a(\Omega)$. 
Hiep \cite{H3} studied in 2010 the convergence in $C_n$-capacity within the class $\mathcal E(\Omega)$.
Recently, Cegrell \cite{Ce3} proved in 2012 that if a sequence of plurisubharmonic functions is bounded from below by a function from the Cegrell class $\mathcal E(\Omega)$ and   convergent in $C_{n-1}$-capacity then the corresponding complex Monge-Amp\`ere measures are convergent in the weak* topology.

The purpose of this paper is to study conditions on a sequence of plurisubharmonic functions which are  equivalent to convergence in   $C_n$-capacity. 
Our main result is the following theorem.\\

\noindent
{\bf Main theorem. }{\em 
Let  $\Omega$  be a bounded hyperconvex domain in $\mathbb C^n$ and let $f\in\mathcal E(\Omega)$, $w\in\mathcal N^a(\Omega,f)$ such that $\int_\Omega (-\rho) (dd^c w)^n<+\infty$   for some $\rho\in \mathcal E_0(\Omega)$.  Assume that       $\{u_j\}\subset \mathcal N^a(\Omega,f)$ such that $u_j\to u_0$ a.e. on $\Omega$ as $j\to+\infty$ and $u_j\geq w$ in $\Omega$ for all $j\geq0$. Then, the following statements are equivalent.

(a)  $ u_j\to u_0$ in  $C_n$-capacity in $\Omega$;

(b)   For every $a>0$, 
we have
$$\lim_{j\to+\infty} \int_{\Omega} \max\left(\frac{u_j}{a},\rho\right) (dd^c u_j)^n= \int_{\Omega} \max\left(\frac{u_0}{a},\rho\right) (dd^c u_0 )^n.$$  

(c) For every $a >0$, we have    
$$\lim_{j\to+\infty} \int_{\Omega} \left[\max\left(\frac{v_j}{a},\rho\right) -\max\left(\frac{u_j}{a},\rho\right) \right] (dd^c u_j)^n= 0,$$   
where 
$v_j:=\left( \sup _{k\geq j} u_{k} \right)^*.$
}\\

The paper is organized as follows. In Section 2 we recall some notions of pluripotential theory.  Section 3 is devoted to the proof of the  main theorem. In Section 4 we   apply  the  main theorem to prove a stability result  for  the solutions of certain complex Monge-Amp\`ere equations.

\section{Preliminaries} 
Some elements of pluripotential theory that will be used throughout the paper can be found
in \cite{ACCH}-\cite{Xing}. 

\begin{definition}{\rm 
Let $n$ be a positive integer. A bounded domain $\Omega$ in $\mathbb C^n$ is called bounded hyperconvex domain if  there exists
a bounded plurisubharmonic function $\varphi: \Omega \to (-\infty,0)$ such that the closure of the set $\{z\in \Omega: \varphi(z)<c\}$  is compact in $\Omega$, for every
$c\in(-\infty,0)$.
}\end{definition} 

We denote by $PSH(\Omega)$ the family of
plurisubharmonic functions defined on $\Omega$ and $PSH^{-}(\Omega)$
denotes the set of negative plurisubharmonic functions on $\Omega$.
By $MPSH(\Omega)$  denotes the set of all maximal  plurisubharmonic functions in $\Omega$.

\begin{definition}{\rm
Let $\Omega$ be a bounded hyperconvex domain in $\mathbb C^n$. 
We say that a bounded, negative  plurisubharmonic function $\varphi$ in
$\Omega$ belongs to $\mathcal E_0(\Omega)$ if $\{\varphi<-\varepsilon\}\Subset \Omega$ for all $\varepsilon>0$ and $\int_\Omega(dd ^c \varphi)^n < +\infty$.

Let $\mathcal F(\Omega)$ be the family of plurisubharmonic functions $\varphi$
defined on $\Omega$, such that there exists a decreasing sequence
$\{\varphi_j\}\subset \mathcal E_0(\Omega)$ that converges pointwise
to $\varphi$ on $\Omega$ as $j\to +\infty$ and
$$\sup_j \int_\Omega (dd^ c \varphi_j)^n<+\infty.$$

We denote by $\mathcal E(\Omega)$ the family of
plurisubharmonic functions $\varphi$ defined on $\Omega$ such that
for every open set $G\Subset\Omega$ there exists a plurisubharmonic function
$\psi\in\mathcal F(\Omega)$ satisfy $\psi=\varphi$ in $G$.

Let $u\in\mathcal E(\Omega)$ and let $\{\Omega_j\}$ be an increasing sequence of bounded hyperconvex domains such that $\Omega_j\Subset\Omega_j\Subset \Omega$ and $\bigcup_{j=1}^{+\infty}\Omega_j =\Omega$. Put 
$$u^j:=\sup\{\varphi\in PSH^-(\Omega): \varphi \leq u \text{ in } \Omega\backslash \Omega_j \}$$
and 
$\mathcal N(\Omega):=\{u\in\mathcal E(\Omega): u^j \nearrow 0 \text{ a.e. in }\Omega\}$.

Let $\mathcal K\in \{ \mathcal F, \mathcal N, \mathcal E\}$. We denote by $\mathcal K^a(\Omega)$  the subclass of $\mathcal K(\Omega)$   such that the Monge-Amp\`{e}re measure
$(dd^c .)^n$ vanishes  on all pluripolar sets of $\Omega$.

Let $f\in\mathcal E(\Omega)$ and $\mathcal K\in \{  \mathcal F^a, \mathcal N^a,  \mathcal E^a,\mathcal F, \mathcal N, \mathcal E\}$. Then we say that a
plurisubharmonic function $\varphi$ defined on $\Omega$ belongs to $\mathcal K(\Omega, f)$ if there exists a function
$\psi\in \mathcal K(\Omega)$ such that
$$\psi+f\leq \varphi \leq f \text{ in }\Omega.$$
}\end{definition}

Now we will show that if  $u\in \mathcal N^a(\Omega,f)$ then the  pluripolar part of  $(dd^c u)^n$ is carried by $\{f=-\infty\}$.

\begin{proposition}\label{pro-1-1-1-1-1}
Let  $\Omega$  be a bounded hyperconvex domain in $\mathbb C^n$. Assume that $f\in  \mathcal E(\Omega)$ and $u\in \mathcal N^a(\Omega,f)$ such that $\int_{\Omega}(-\rho)(dd^c u)^n<+\infty$ for some $\rho\in\mathcal E_0(\Omega)$. Then 
$$1_{\{u=-\infty\}}(dd^c u)^n=1_{\{f=-\infty\}}(dd^c f)^n \text{ in }\Omega.$$
\end{proposition}

\begin{proof}
Let $v\in \mathcal F^a(\Omega)$ such that
$v+f\leq u\leq f$ in $\Omega$. By Lemma 4.1 and Lemma 4.12 in \cite{ACCH} we have
\begin{align*}
1_{\{f=-\infty\}}(dd^c f)^n
&\leq 1_{\{u=-\infty\}}(dd^c u)^n 
\\&\leq 1_{\{v+f=-\infty\}}(dd^c (v+f))^n
=1_{\{f=-\infty\}}(dd^c f)^n.
\end{align*}
It follows that  
$$1_{\{u=-\infty\}}(dd^c u)^n=1_{\{f=-\infty\}}(dd^c f)^n \text{ in }\Omega.$$
The proof is complete.
\end{proof}

\begin{proposition}\label{Hong}
Let  $\Omega$  be a bounded hyperconvex domain in $\mathbb C^n$. Let  $f\in  \mathcal E(\Omega)$ and $u\in \mathcal N^a(\Omega,f)$ such that $\int_{\Omega}(-\rho)(dd^c u)^n<+\infty$ for some $\rho\in\mathcal E_0(\Omega)$. Assume that
$v\in \mathcal E(\Omega)$ such that $v\leq f$ and $(dd^c v)^n \geq (dd^c u)^n$ in $\Omega$. 
Then  $v\leq u$ on $\Omega$.
\end{proposition}

\begin{proof}
Since the measure $1_{\{u>-\infty\}} (dd^c u)^n$ vanishes on all pluripolar subsets of $\Omega$, by Proposition 4.3 in \cite{KH} we get
$$(dd^c \max (u,v) )^n \geq  1_{\{u>-\infty\}} (dd^c u)^n.$$
Hence, 
$$ 1_{\{ \max (u,v) >-\infty\}} (dd^c \max (u,v) )^n \geq  1_{\{u>-\infty\}} (dd^c u)^n.$$
Moreover, by the hypotheses and Proposition \ref{pro-1-1-1-1-1} we have 
$$1_{\{ \max (u,v)  =-\infty\}} (dd^c  \max (u,v) )^n = 1_{\{u=-\infty\}} (dd^c u)^n.$$
Hence, $(dd^c  \max (u,v) )^n \geq(dd^c u)^n$ in $\Omega$. Therefore, from Theorem 3.6 in \cite{ACCH} it follows that $\max (u,v)=u$ in $\Omega$. Thus, $v\leq u$ in $\Omega$. The proof is complete.
\end{proof}

\section{Proof of  the main theorem}

In order to prove the main theorem, we need the following auxiliary lemmas.

\begin{lemma} \label{le1}
Let  $\Omega$  be a bounded hyperconvex domain in $\mathbb C^n$  and let $f\in  \mathcal E(\Omega)$. Assume that $\rho\in\mathcal E_0(\Omega)$ and $u \in  \mathcal N^a(\Omega,f)$   such that  $\int_{\Omega}(-\rho)(dd^c u)^n<+\infty$. Then for every $v\in\mathcal E^a(\Omega,f)$  and for every $\varphi \in \mathcal E_0(\Omega)$ with $\varphi \geq \rho $, we have
\begin{align*}
\frac{1}{n!}\int_{\{u <v\}}(v- u)^n(dd^c \varphi )^n&+\int_{\{u<v\}}-\varphi (dd^c v)^n
\\&\leq \int_{\{u<v\}}- \varphi (dd^cu)^n.
\end{align*}
\end{lemma}

\begin{proof}
For $j\in\mathbb N^*$, put $v_j=\max(u ,v-\frac{1}{j})$.
Because  $u\leq  v_j\leq f$ in $\Omega$, we have $v_j\in\mathcal F^a(\Omega,f)$.  
By Lemma 3.5 in \cite{ACCH} we have
\begin{align*}
\frac{1}{n!}\int_{\Omega}(v_j-u)^n(dd^c \varphi)^n + \int_{\Omega}- \varphi(dd^cv_j)^n \leq \int_{\Omega}-\varphi (dd^c u)^n.
\end{align*}
By Theorem 4.1 in \cite{KH} we have $v_j=v-\frac{1}{j}$ in $\{u<v_j\}$. Hence,
\begin{align*}
&\frac{1}{n!}\int_{\{u <v_j\}}(v_j-u)^n(dd^c \varphi)^n +\int_{\{u <v_j\}} -\varphi (dd^cv)^n
\\&= \frac{1}{n!}\int_{\{u  <v_{j} \}}(v_j - u)^n(dd^c\varphi)^n + \int_{\{u < v_{j}\}} -\varphi (dd^cv_j)^n
\\&\leq \frac{1}{n!}\int_{\Omega}(v_j -u )^n(dd^c \varphi )^n +\int_{\{u< v_j\}} -\varphi (dd^cv_j)^n
\\&\leq \frac{1}{n!}\int_{\Omega}(v_j - u)^n(dd^c \varphi )^n+\int_{\Omega}- \varphi (dd^c v_j)^n
-\int_{\{u =v_j\}}-\varphi (dd^c v_j)^n
\\&\leq \int_{\Omega}- \varphi (dd^c u)^n -\int_{\{u\geq v\}}- \varphi (dd^c v_j )^n.
\end{align*}
Now, since $u=v_j$ in $\{u>v-\frac{1}{j}\}$ so by Theorem 4.1 in \cite{KH} imply that  
$$(dd^c u)^n=(dd^c v_j)^n \text{ in  }\{u\geq v\}\cap \{u>-\infty\}.$$
Moreover, by Proposition \ref{pro-1-1-1-1-1} we have
$$
1_{\{u =-\infty\}}(dd^c u)^n =1_{\{v_j=-\infty\}}(dd^cv_j)^n =1_{\{f=-\infty\}}(dd^c f)^n 
\text{ in }\Omega.$$ 
Hence, we obtain that 
$$(dd^c u)^n=(dd^c v_j)^n \text{ in  }\{u\geq v\}.$$
Therefore,  
\begin{align*}
&\frac{1}{n!}\int_{\{u <v_j\}}(v_j-u)^n (dd^c \varphi )^n+\int_{\{u<v_j\}} - \varphi (dd^c v)^n
\\&\leq   \int_{\Omega}- \varphi (dd^cu )^n-\int_{\{u \geq v\}}-\varphi (dd^c u)^n
\\&= \int_{\{u<v\}}- \varphi (dd^c u )^n .
\end{align*}
Let $j\to+\infty$ we obtain that
\begin{align*}
\frac{1}{n!}\int_{\{u <v\}}(v-u)^n(dd^c \varphi )^n +\int_{\{u <v\}}- \varphi (dd^cv)^n\leq \int_{\{u <v\}}-\varphi (dd^cu )^n.
\end{align*} 
The proof is complete.
\end{proof}

\begin{lemma}\label{le2}

Let  $\Omega$  be a bounded hyperconvex domain in $\mathbb C^n$  and
let $\{u_j\}\subset \mathcal E^a(\Omega)$  such that $u_j\geq u_1$ for
every $j\geq1$ and $u_j\to u_0$ in $C_{n}$-capacity in $\Omega$. 
Assume that  
$\{\varphi^k_j\}$, $k=1,2$ are  sequences  of  uniformly bounded
plurisubharmonic functions in $\Omega$ which converges weakly to a plurisubharmonic function  $\varphi^k_0$ in $\Omega$. Then
$\varphi_j^1 \varphi_j^2 (dd^c u_j)^n \to \varphi^1_0 \varphi^2_0(dd^c u_0)^n$ weakly as $j\to+\infty$.
\end{lemma}

\begin{proof}
Without loss of generality we can assume that $u_j\in\mathcal F^a(\Omega)$ and $-1\leq \varphi_j^k\leq 0$ in $\Omega$ for all $j\geq 0$, $k=1,2$. Put
$$\psi_j^1=\frac{(\varphi_j^1+ \varphi^2_j+2)^2 +4}{2},\ \psi^2_j=\frac{(\varphi_j^1+2)^2}{2} \text{ and } \psi^3_j=\frac{(\varphi^2_j+2)^2}{2}.$$
It is clear that $\psi_j^k\in PSH(\Omega)$, $0\leq \psi_j^k\leq 4$ and $\psi^k_j\to \psi^k_0$ weakly in $\Omega$ as $j\to+\infty$, $k=1,2,3$.
Since $\varphi_j^1\varphi_j^2=\psi^1_j -\psi^2_j-\psi^3_j$ in $\Omega$ we obtain by Theorem   3.4 in \cite{Xing} that 
\begin{align*}
\varphi_j^1 \varphi_j^2 (dd^c u_j)^n 
&=\psi^1_j (dd^c u_j)^n 
-\psi^2_j(dd^c u_j)^n 
-\psi^3_j (dd^c u_j)^n 
\\&\to \psi^1_0 (dd^c u_0)^n 
-\psi^2_0(dd^c u_0)^n 
-\psi^3_0 (dd^c u_0)^n 
=\varphi^1_0 \varphi^2_0(dd^c u_0)^n
\end{align*}
weakly in $\Omega$ as $j\to+\infty$. The proof is complete.
\end{proof}

\begin{proof}[Proof of the main theorem] 
Without loss of generality we can assume that $f<0$ and $-1\leq \rho\leq 0$ in $\Omega$.

(a)$\Rightarrow$(b).  Fix $a>0$. Put
$$\varphi_j:=\max\left(\frac{u_j}{a},\rho\right).$$
Because 
\begin{align*}
0\leq \sup_j \int_\Omega -\varphi_j (dd^c u_j)^n 
\leq \int_\Omega -\rho (dd^c w)^n<+ \infty,
\end{align*}
it remains to prove that there exists a subsequence $\{u_{j_k}\}$ of sequence $\{u_j\}$ such that  
$$
\lim_{k\to+\infty} \int_{\Omega} \varphi_{j_k} (dd^c u_{j_k})^n= \int_{\Omega} \varphi_0 (dd^c u_0)^n.
$$

First we claim that 
there exists an increasing sequence $\{j_k\}\subset\mathbb N^*$ such that
\begin{equation}\label{eq0001}
\lim_{k\to+\infty}  \int_{\Omega} \varphi_{j_k}  \max\left(1+\frac{u_{j_k}}{k},0\right) (dd^c u_{j_k})^n
=\int_{\{u_0>-\infty\}}  \varphi_0 (dd^c u_0)^n.
\end{equation}
Indeed, let  $\chi_k\in\mathcal C^\infty_0(\Omega)$ such that $0\leq \chi_k\leq \chi_{k+1}\leq 1$ in $\Omega$, $\{\rho\leq -\frac{1}{k}\}\Subset \{\chi_k=1\}$ and
$$\int_{\{\chi_k<1\}} (-\rho) (dd^c u_0)^n \leq \frac{1}{k}.$$  
Since $u_j\to u_0$ in $C_n$-capacity in $\Omega$ as $j\to+\infty$, so
$\max(u_j,-k)\to \max(u_0,-k)$ in $C_n$-capacity as $j\to+\infty$. By
Lemma \ref{le2}  we have
\begin{align*}
 \varphi_j  \max\left(1+\frac{u_{j}}{k},0\right) (dd^c \max(u_j,-k))^n
\to  \varphi_0 \max\left(1+\frac{u_{0}}{k},0\right) (dd^c \max(u_0,-k))^n
\end{align*}
weakly in $\Omega$ as $j\to+\infty$. Hence, by
Theorem 4.1 in \cite{KH} we get
\begin{align*}
&\lim_{j\to+\infty} \int_\Omega  \chi _k \varphi_j  \max\left(1+\frac{u_{j}}{k},0\right) (dd^c u_j)^n
\\&=\lim_{j\to+\infty} \int_\Omega \chi _k \varphi_j \max\left(1+\frac{u_{j}}{k},0\right) (dd^c \max(u_j,-k))^n
\\&= \int_\Omega \chi_k  \varphi_0  \max\left(1+\frac{u_{0}}{k},0\right) (dd^c \max(u_0,-k))^n
\\&= \int_\Omega \chi _k \varphi_0  \max\left(1+\frac{u_{0}}{k},0\right) (dd^c u_0)^n.
\end{align*} 
Because  $\chi_k\max\left(1+\frac{u_{0}}{k},0\right) \nearrow 1_{\{u_0>-\infty\}}$ as $k\to+\infty$ in $\Omega$,  
we have
\begin{align*}
\lim_{k\to+\infty}\int_\Omega \chi _k \varphi_0\max\left(1+\frac{u_{0}}{k},0\right)  (dd^c u_0)^n
=\int_{\{u_0>-\infty\}}  \varphi_0 (dd^c u_0)^n.
\end{align*}
Therefore,   there exists an increasing sequence $\{j_k\}\subset\mathbb N^*$ such that
\begin{equation}\label{eq--0001}
\lim_{k\to+\infty}  \int_{\Omega}\chi _k \varphi_{j_k}  \max\left(1+\frac{u_{j_k}}{k},0\right) (dd^c u_{j_k})^n
=\int_{\{u_0>-\infty\}}   \varphi_0 (dd^c u_0)^n.
\end{equation} 
Now, fix $k_0\in \mathbb N^*$. By  
the proof of the theorem in \cite{Ce3} (see (3.1) in \cite{Ce3})
we have
\begin{equation}\label{eq--0001=1}
\begin{split}
&\liminf_{k\to+\infty}  \int_{\Omega} (1-\chi_k) \varphi_{j_k}  \max\left(1+\frac{u_{j_k}}{k},0\right)  (dd^c u_{j_k})^n
\\&\geq \liminf_{k\to+\infty}\int_{\Omega} (1-\chi_{k})\varphi_{j_k}  (dd^c u_{j_k})^n
\geq\liminf_{k\to+\infty} \int_{\Omega} (1-\chi_{k_0})\rho  (dd^c u_{j_k})^n
\\&=\liminf_{k\to+\infty}\left [\int_{\Omega} \rho  (dd^c u_{j_k})^n
-\int_{\Omega} \chi_{k_0}\rho  (dd^c u_{j_k})^n \right ]
\\& =\int_{\Omega} \rho  (dd^c u_0)^n
-\int_{\Omega} \chi_{k_0}\rho  (dd^c u_0)^n
\\&\geq \int_{\{\chi_{k_0}<1\}} \rho (dd^c u_0)^n \geq -\frac{1}{k_0}.
\end{split}
\end{equation}
Combining this with \eqref{eq--0001}  we arrive at 
\begin{align*}
&\lim_{k\to+\infty}  \int_{\Omega} \varphi_{j_k} \max\left(1+\frac{u_{j_k}}{k},0\right) (dd^c u_{j_k})^n\\&
=\lim_{k\to+\infty}  \int_{\Omega} \chi_k \varphi_{j_k}  \max\left(1+\frac{u_{j_k}}{k},0\right) (dd^c u_{j_k})^n
\\&
=\int_{\{u_0>-\infty \}}  \varphi_0 (dd^c u_0)^n.
\end{align*}
This proves the claim. 

The measure $1_{\{u_{j_k}>-\infty\}} \varphi_{j_k} \max \left( \frac{u_{j_k}}{k},-1\right) (dd^c u_{j_k})^n$ vanishes on all pluripolar subset of $\Omega$, hence by Lemma 5.14 in \cite{Ce04}  there exists
$h_{k}\in\mathcal F^a(\Omega)$ such that
$$(dd^c h_k)^n =1_{\{u_{j_k}>-\infty\}} \varphi_{j_k} \max \left( \frac{u_{j_k}}{k},-1\right) (dd^c u_{j_k})^n.$$
Because  $(dd^c h_k)^n \leq (dd^c u_{j_k})^n$ in $\Omega$ and the measure  $(dd^c h_k)^n$ vanishes on all pluripolar subset of $\Omega$, from  Corollary 3.2 in \cite{ACCH} we have  
$$h_k\geq u_{j_k}\geq w\text{ in }\Omega.$$

We claim that $h_k\to 0$ in $C_n$-capacity in $\Omega$. Indeed,  let $\delta>0$ and $\psi \in PSH(\Omega)$ with $-1\leq \psi \leq 0$.
By Theorem 3.1 in \cite{ACCH} we have
\begin{align*}
\int_{\{h_k<-\delta\}} (dd^c\psi )^n
&\leq   \int_{\{h_k<\delta \psi \}} (dd^c  \psi )^n
\leq \frac{1}{\delta^n} \int_{\{h_k<\delta \psi \}} (dd^c h_k)^n
\\&\leq \frac{1}{\delta^n} \int_{\{u_{j_k}>-\infty\} } \varphi_{j_k}  \max \left( \frac{u_{j_k}}{k},-1\right)  (dd^c u_{j_k})^n
\\&\leq -\frac{1}{\delta^n} \int_{\{u_{j_k}>-\infty\} }   \max \left( \frac{u_{j_k}}{k},\rho \right) (dd^c u_{j_k})^n.
\end{align*}
Therefore, by Lemma 3.3 in \cite{ACCH} and Proposition \ref{pro-1-1-1-1-1} we obtain that 
\begin{align*}
\int_{\{h_k<-\delta\}} (dd^c\psi )^n
&\leq -\frac{1}{\delta^n} \int_{\Omega }   \max \left( \frac{u_{j_k}}{k},\rho \right) (dd^c u_{j_k})^n 
+ \frac{1}{\delta^n} \int_{\{u_{j_k}=-\infty\}}\rho (dd^c u_{j_k})^n
\\&\leq -\frac{1}{\delta^n} \int_{\Omega }  \max \left( \frac{w}{k},\rho \right) (dd^c w)^n 
+ \frac{1}{\delta^n} \int_{\{w=-\infty\}}\rho (dd^c w)^n 
\\&= -\frac{1}{\delta^n} \int_{\{w>-\infty\}}  \max \left( \frac{w}{k},\rho \right) (dd^c w)^n .
\end{align*}
It follows that
\begin{align*}
 C_n(\{h_k<-\delta\})
 \leq  -\frac{1}{\delta^n}  \int_{\{w>-\infty\}} \max \left( \frac{w}{k},\rho \right) (dd^c w)^n.
\end{align*}
Hence, we get
$$\lim_{k\to+\infty} C_n(\{h_k<-\delta\})=0,$$
for every $\delta>0$. 
Thus, $h_k\to 0$ in $C_n$-capacity in $\Omega$ as $k\to+\infty$. This proves the claim, and therefore,  by \eqref{eq--0001=1} and  the Theorem in  \cite{Ce3} we have 
\begin{align*}
0&\leq \limsup_{k\to+\infty} \int_{ \{u_{j_k} >-\infty\}}  \varphi_{j_k}  \max \left( \frac{u_{j_k}}{k},-1 \right)  (dd^c u_{j_k})^n  
\\&\leq \limsup_{k\to+\infty} \int_{ \Omega} (1-\chi_{k_0}) (-\rho) (dd^c u_{j_k})^n  
+  \limsup_{k\to+\infty} \int_{\Omega} \chi_{k_0}  (dd^c h_{k})^n
\\& \leq \frac{1}{k_0},
\end{align*}
for all $k_0\in\mathbb N^*$.
Thus,
$$
\lim_{k\to+\infty} \int_{ \{u_{j_k} >-\infty\}}  \varphi_{j_k}\max \left( \frac{u_{j_k}}{k},-1 \right) (dd^c u_{j_k})^n=0.
$$
Combining this with \eqref{eq0001}  we arrive at
$$\lim_{k\to+\infty} \int_{\{u_{j_k}>-\infty\}} \varphi_{j_k} (dd^c u_{j_k})^n= \int_{\{u_0>-\infty\}} \varphi_0 (dd^c u_0)^n.$$ 
Moreover, by Proposition \ref{pro-1-1-1-1-1}, we have 
\begin{align*}  \int_{\{u_{j_k}=-\infty\}} \varphi_{j_k} (dd^c u_{j_k})^n
=  \int_{\{f=-\infty\}} \rho (dd^c f)^n
= \int_{\{u_0=-\infty\}} \varphi_0 (dd^c u_0)^n.
\end{align*}
Hence, we get
$$\lim_{k\to+\infty} \int_{\Omega} \varphi_{j_k} (dd^c u_{j_k})^n= \int_{\Omega} \varphi_0 (dd^c u_0)^n.$$ 

(b)$\Rightarrow$(c). Fix $a>0$.
Since $u_j\to u_0$ a.e. in $\Omega$ as $j\to+\infty$ so $v_j\searrow u_0$ as $j\nearrow +\infty$. Hence, $v_j\to u_0$ in $C_n$-capacity in $\Omega$. Therefore, by the proof of (a)$\Rightarrow$(b) and Lemma 3.3 in \cite{ACCH}, we have 
\begin{align*}
\int_{\Omega} \max\left (\frac{u_0}{a} , \rho \right ) (dd^c u_{0})^n
&=\lim_{j\to+\infty} \int_{\Omega} \max\left (\frac{v_j}{a} , \rho \right ) (dd^c v_{j})^n    
\\&\geq \lim_{j\to+\infty} \int_{\Omega} \max\left (\frac{v_j}{a} , \rho \right ) (dd^c u_{j})^n 
\\&\geq \lim_{j\to+\infty} \int_{\Omega} \max\left (\frac{u_j}{a} , \rho \right ) (dd^c u_{j})^n 
\\&= \int_{\Omega} \max\left (\frac{u_0}{a} , \rho \right ) (dd^c u_{0})^n.   
\end{align*}
It follows that 
$$\lim_{j\to+\infty} \int_{\Omega} \max\left (\frac{v_j}{a} , \rho \right ) (dd^c u_{j})^n 
= \int_{\Omega} \max\left (\frac{u_0}{a} , \rho \right ) (dd^c u_{0})^n.   $$
Therefore, we obtain that
\begin{align*}
\lim_{j\to+\infty} \int_{\Omega} \left[\max\left(\frac{v_j}{a},\rho\right) -\max\left(\frac{u_j}{a},\rho\right) \right] (dd^c u_j)^n
=0.
\end{align*}

(c)$\Rightarrow$(a).  Because $v_j \searrow u_0$ in $\Omega$ as $j\nearrow +\infty$, we get $v_j\to u_0$ in $C_n$-capacity in $\Omega$. Hence, it is sufficient to prove that $v_j-u_j\to 0$ in $C_n$-capacity in $\Omega$. Let $K$ be a compact subset of $\Omega$ and let $\varepsilon, \delta>0$. Without loss of generality we can assume that 
  $K\Subset \{\rho= -1\}$.  
Choose $\chi\in\mathcal C^\infty_0(\Omega)$ and  $a>b>1$  such that $0\leq \chi\leq 1$,   $\{\rho \leq -\varepsilon\}\Subset \{\chi=1\}$,
 $\{\chi\neq 0\} \subset \{a\rho<-b \}$ and 
\begin{equation}\label{eq06}
\frac{a}{b} \int_{\{w>-\infty\}} -\max \left (\frac{w}{a},  \rho \right ) (dd^c w)^n <\varepsilon .
\end{equation}
Let  $\psi_j \in \mathcal E_0(\Omega)$ with  $\psi_j\geq \rho$ such that
\begin{equation}\label{eq03}
C_n(K\cap \{v_{j}-u_{j}>2\delta\})
<\int_{K\cap  \{v_{j}-u_{j}>2\delta\} } (dd^c \psi_j)^n+ \varepsilon.
\end{equation}
Note that  $u_j\leq v_j$ in $\Omega$ for all $j\geq1$. From the hypotheses we have
\begin{align*}
0& \leq  \limsup_{j\to+\infty} \int_{\{u_{j}<v_j-\delta\}\cap \{u_j >-b\}}  \chi  (dd^c u_{j})^n
\\& \leq  \frac{1}{\delta} \limsup_{j\to+\infty}  \int_{\{u_{j}<v_j-\delta\}\cap \{u_j >-b\}} \chi (v_j-u_{j})  (dd^c u_{j})^n
\\& \leq \frac{a}{\delta} \limsup_{j\to+\infty} \int_{ \Omega}  \left [ \max\left (\frac{v_j}{a} ,\rho  \right ) - \max\left (\frac{u_j}{a} ,\rho  \right ) \right]  (dd^c u_{j})^n 
\\& 
=0.
\end{align*}
It follow that
$$\lim_{j\to+\infty} \int_{\{u_{j}<v_j-\delta\}\cap \{u_j>-b\}}  \chi  (dd^c u_{j})^n=0.$$
By Lemma 3.3 in \cite{ACCH} and Proposition \ref{pro-1-1-1-1-1}   we  have
\begin{align*}
& \limsup_{j\to+\infty} \int_{\{u_{j}<v_j-\delta\}\cap \{u_j>- \infty \}} \chi  (dd^c u_{j})^n
\\&=  \limsup_{j\to+\infty} \int_{\{u_{j}<v_j-\delta\}\cap \{-\infty< u_j\leq-b\}} \chi  (dd^c u_{j})^n 
\\&\leq    \limsup_{j\to+\infty} \int_{\{u_j>-\infty \}} - \max\left (\frac{u_j}{b} ,  \frac{a \rho}{b}  \right ) (dd^c u_{j})^n 
\\&\leq  \frac{a}{b} \limsup_{j\to+\infty} \int_{\Omega} -\max\left (\frac{u_j}{a} , \rho \right ) (dd^c u_{j})^n  +  \frac{a}{b} \int_{\{u_j=-\infty\}} \max\left (\frac{u_j}{a} ,  \rho \right )  (dd^cu_j)^n 
\\&\leq \frac{a}{b}  \int_{\Omega} -\max\left (\frac{w}{a} , \rho \right ) (dd^c w)^n  +  \frac{a}{b} \int_{\{w=-\infty\}} \max\left (\frac{w}{a} ,  \rho \right )  (dd^c w)^n 
\\&= \frac{a}{b}  \int_{\{w>-\infty\}} -\max \left (\frac{w}{a}, \rho \right ) (dd^c w)^n.
\end{align*}  
Therefore, by \eqref{eq06}  we get
\begin{equation}\label{eq04}
 \limsup_{j\to+\infty} \int_{\{u_{j}<v_j-\delta\}\cap \{u_j>-\infty\}} \chi  (dd^c u_{j})^n<\varepsilon.
\end{equation}
Now, by Proposition \ref{pro-1-1-1-1-1} and  Lemma \ref{le1} we have    
\begin{align*}
\int_{K\cap  \{v_{j}-u_{j}>2\delta\}} (dd^c \psi_j)^n
& \leq \int_{\{u_{j}<v_{j} -2\delta\}}(dd^c\psi_j)^n
\\&\leq \frac{1}{\delta^n}\int_{\{u_{j}< v_{j}-2\delta\}} (v_{j}- \delta -u_{j})^n (dd^c \psi _j )^n
\\&\leq \frac{1}{\delta^n}\int_{\{u_{j}<v_{j}-\delta\}} (v_{j}- \delta -u_{j})^n  (dd^c \psi _j)^n
\\& \leq \frac{n!}{\delta^n}\int_{\{u_{j} <v_{j}-\delta\}\cap \{u_j>-\infty\}} -\psi_j (dd^c u_{j})^n.
\end{align*}
Hence, from \eqref{eq04} we obtain that 
\begin{align*}
\limsup_{j\to+\infty} \int_{K\cap  \{v_{j}-u_{j}>2\delta\}} (dd^c \psi_j)^n
& \leq \frac{n!}{\delta^n} \limsup_{j\to+\infty}   \int_{\{u_{j} <v_{j}-\delta\}\cap \{u_j>-\infty\}} -\psi_j (dd^c u_{j})^n
\\& \leq \frac{n!}{\delta^n} \limsup_{j\to+\infty}   \int_{\Omega} -\psi_j (1-\chi) (dd^c u_{j})^n
\\& + \frac{n!}{\delta^n} \limsup_{j\to+\infty}   \int_{\{u_{j} <v_{j}-\delta\}\cap \{u_j>-\infty\}} \chi (dd^c u_{j})^n
\\& \leq \frac{n!}{\delta^n}  \limsup_{j\to+\infty}  \int_{\{\rho>-\varepsilon\}} -\rho   (dd^c u_{j})^n + \frac{n! \varepsilon}{\delta^n}
\\& \leq \frac{n!}{\delta^n} \limsup_{j\to+\infty}   \int_{\Omega} -\max(\rho,-\varepsilon)   (dd^c u_{j})^n + \frac{n! \varepsilon}{\delta^n}
\\& \leq \frac{n!}{\delta^n}  \int_{\Omega} -\max(\rho,-\varepsilon)   (dd^c w)^n + \frac{n! \varepsilon}{\delta^n}.
\end{align*}
Combining this with    \eqref{eq03} we get
\begin{align*}
\limsup_{j\to+\infty} C_n(\{K\cap  \{v_{j}-u_{j}>2\delta\}) 
\leq \frac{n!}{\delta^n}  \int_{\Omega} -\max(\rho,-\varepsilon)   (dd^c w)^n + \left( \frac{n! }{\delta^n} + 1 \right)\varepsilon.
\end{align*}
Let $\varepsilon \searrow 0$ we obtain that
\begin{align*}
\lim_{j\to+\infty} C_n(\{K\cap  \{v_{j}-u_{j}>2\delta\}) =0.
\end{align*}
Thus, $v_j-u_j\to 0$ in $C_n$-capacity in $\Omega$.
The proof is complete.
\end{proof}

\section{Application}
In this section, we prove a generalization of Cegrell and Ko\l odziej's stability theorem from \cite{CK}.
 First, we need the following. 

\begin{lemma}\label{le4.1111}
Let  $\Omega$  be a  bounded
hyperconvex domain in $\mathbb C^n$ and let $f\in\mathcal
E(\Omega)$, $w\in\mathcal N^a(\Omega,f)$ such that  $\int_{\Omega}(-\rho) (dd^c w)^n<+\infty$  for some $\rho\in\mathcal E_0(\Omega)$.  
Then for every   
nonnegative Borel measures $\mu$ in $\Omega$ such that
$$(dd^c f)^n \leq \mu\leq (dd^c w)^n,$$
there exists a unique  $u\in \mathcal N^a(\Omega,f)$ such that  $u\geq w$ and  $(dd^c u)^n=\mu$ in $\Omega$.
\end{lemma}
\begin{proof} The uniqueness imply from Proposition \ref{Hong}.  From the hypotheses and Proposition \ref{pro-1-1-1-1-1} we have 
$$1_{ \{w=-\infty\}} \mu 
=1_{ \{f=-\infty\}} (dd^c f)^n  \text{ in } \Omega.$$
Let $\{\Omega_j\}$ be a sequence of bounded hyperconvex domains such that $\Omega_j \Subset\Omega_{j+1}\Subset \Omega$ and $ \Omega =\bigcup_{j=1}^{+\infty} \Omega_j$. Because the measure $1_{ \{w>-\infty\}} \mu $ vanishes on all pluripolar subsets of $\Omega$, applying  Proposition 5.1 in \cite{HTH} we see that there are  $u_j \in \mathcal N^a(\Omega_j,f)$ such that  
$$(dd^c u_j)^n = 1_{\Omega_j \cap  \{w>-\infty\}} \mu 
+1_{ \Omega_j \cap  \{f=-\infty\}} (dd^c f)^n
 = \mu \text{ in } \Omega_j .$$
By Proposition \ref{Hong} we have $w\leq u_{j+1} \leq u_j \leq f$  on $ \Omega_j$. 
Put $u:=\lim_{j\to+\infty} u_j$. Then $w\leq u\leq f$ and $(dd^c u)^n =\mu$ in $\Omega$.
Moreover, since $w\in \mathcal N^a(\Omega,f)$, we get 
$u \in \mathcal N^a(\Omega,f)$. The proof is complete. 
\end{proof}

\begin{proposition}\label{pro4.1}
Let  $\Omega$  be a bounded
hyperconvex domain in $\mathbb C^n$ and let $f\in\mathcal
E(\Omega)$. Assume that $w\in\mathcal N^a(\Omega,f)$ such that  $\int_{\Omega}(-\rho) (dd^c w)^n<+\infty$  for some $\rho\in\mathcal E_0(\Omega)$. Then for every  sequence of
nonnegative Borel measures $\{\mu_j\}$ that  converges weakly to a non-negative Borel measure $\mu_0$ in $\Omega$ and satisfies 
$$(dd^c f)^n \leq \mu_j\leq (dd^c w)^n \text{ for all }j\geq 0,$$
there exist unique  $u_j\in \mathcal N^a(\Omega,f)$ such that  $u_j\geq w$, $(dd^c u_j)^n=\mu_j$ and $u_j\to u_0$ in $C_n$-capacity in $\Omega$.
\end{proposition}

\begin{proof} 
By Lemma \ref{le4.1111} there exist unique $u_j\in \mathcal N^a(\Omega,f)$ such that $u_j\geq w$ and $(dd^c u_j)^n=\mu_j$ in $\Omega$. 
Since $u_j\geq w$,  the sequence $\{u_j\}$  is compact in $ L^1_{loc}(\Omega)$. 
Let $u$ be a cluster point and let $\{u_{j_k} \}$ be a subsequence of the sequence $\{u_{j} \}$ such that $u_{j_k}\to u$ a.e. in $\Omega$.   Put
$v_k:=\left( \sup_{l\geq k} u_{j_l} \right)^*.$ 
We claim that
\begin{equation}\label{eq-1}
\lim_{k\to+\infty} \int_{\Omega} \left[ \max \left( \frac{v_k}{a}, \rho\right) -\max \left( \frac{u_{j_k}}{a}, \rho\right)  \right]   (dd^c u_{j_k} )^n=0,
\end{equation} 
for every $a>0$. Indeed,  let $\varepsilon>0$. Choose  $\chi\in \mathcal C^\infty_0(\Omega)$ such that $0\leq \chi \leq 1$ and
$\{\chi=1\}\subset \{\rho<-\varepsilon\}$. By Proposition \ref{pro-1-1-1-1-1} we have that the measure $ 1_{\{f>-\infty\}} \chi (dd^c w)^n$  vanishes on all pluripolar subsets of $\Omega$. 
By Lemma 3.1 in \cite{Ce3} we get
\begin{align*}
0&\leq \limsup_{k\to+\infty} \int_{\Omega} \left[ \max \left( \frac{v_k}{a}, \rho\right) -\max \left( \frac{u_{j_k}}{a}, \rho\right)  \right] \chi  (dd^c u_{j_k} )^n
\\&= \limsup_{k\to+\infty} \int_\Omega \left[ \max \left( \frac{v_k}{a}, \rho\right) -\max \left( \frac{u_{j_k}}{a}, \rho\right)  \right]  1_{\{f>-\infty\}} \chi (dd^c u_{j_k} )^n
\\&\leq \limsup_{k\to+\infty} \int_\Omega \left[ \max \left( \frac{v_k}{a}, \rho\right) -\max \left( \frac{u_{j_k}}{a}, \rho\right)  \right]  1_{\{f>-\infty\}} \chi (dd^c w )^n
\\&=\int_\Omega \left[ \max \left( \frac{u}{a}, \rho\right) -\max \left( \frac{u}{a}, \rho\right)  \right] 1_{\{f>-\infty\}} \chi (dd^c w )^n 
\\&=0.
\end{align*}
It follows that 
\begin{align*}
0&\leq \limsup_{k\to+\infty} \int_{\Omega} \left[ \max \left( \frac{v_k}{a}, \rho\right) -\max \left( \frac{u_{j_k}}{a}, \rho\right)  \right]   (dd^c u_{j_k} )^n
\\&\leq \limsup_{k\to+\infty} \int_{\Omega} \left[ \max \left( \frac{v_k}{a}, \rho\right) -\max \left( \frac{u_{j_k}}{a}, \rho\right)  \right] (1- \chi)  (dd^c u_{j_k} )^n
\\&\leq \limsup_{k\to+\infty} \int_{\{\rho\geq -\varepsilon\}} \left[ -\max \left( \frac{v_k}{a}, \rho\right) -\max \left( \frac{u_{j_k}}{a}, \rho\right)  \right]   (dd^c u_{j_k} )^n
\\& \leq 2 \int_{\Omega} -\max(\rho,-\varepsilon) (dd^c w)^n.
\end{align*}
Let $\varepsilon\searrow 0$ we obtain \eqref{eq-1}.
This proves the claim, and therefore, by the main theorem we get 
$u_{j_k}\to u$ in $C_n$-capacity in $\Omega$ as $k\to+\infty$. Hence, by \cite{Ce3} we have $(dd^c u)^n=\mu_0$ in $\Omega$.
It is clear that $u\in \mathcal N^a(\Omega,f)$. 
From the uniqueness  of $u_0$ we get 
$u=u_0$. Thus, $u_{j_k}\to u_0$ a.e. in $\Omega$.   It follows that  $u_j\to u_0$ a.e. in $\Omega$.
Similarly, we get
$$\lim_{j\to+\infty} \int_{\Omega} \left[ \max \left( \frac{v_j}{a}, \rho\right) -\max \left( \frac{u_{j}}{a}, \rho\right)  \right]   (dd^c u_{j} )^n=0,$$
for every $a>0$, where $v_j:=(\sup_{k\geq j} u_k)^*.$
Now, again by the  main theorem  we get  $u_j\to u_0$ in $C_n$-capacity in $\Omega$.
The proof is complete.
\end{proof}


\end{document}